\documentclass{article}

\usepackage{epigraph}

\usepackage{etoolbox}

\makeatletter
\patchcmd{\epigraph}{\@epitext{#1}}{\itshape\@epitext{#1}}{}{}
\makeatother

\usepackage[english]{babel}
\usepackage[dvipsnames,usenames]{xcolor}
\definecolor{mygray}{gray}{.7}

\usepackage[demo]{graphicx}
\usepackage{tikz-qtree}

\usepackage{stackengine}
\usepackage{scalerel}

\stackMath

\newcommand{\nats}                  {{\mathbb N}}

\newcommand{\BMt} {\mathbb{B}}
\newcommand{\CMt} {\mathbb{C}}

\newcommand{\Mt} {\mathbb{M}}

\usepackage{verbatim,graphics,indentfirst} 
\usepackage{url}
\usepackage{stmaryrd}
\usepackage[nosumlimits]{amsmath}
\usepackage{amssymb}
\usepackage{amsthm}
\usepackage{latexsym}
\usepackage{enumerate}
\usepackage{graphicx}
\usepackage[T1]{fontenc}
\usepackage[breaklinks]{hyperref}
\usepackage{mathrsfs} 
\usepackage{setspace}
\usepackage[all]{xy}

 \newcommand{\mc}{\mathsf{mc}}

 \newcommand{\scc}{\mathsf{sc}}
  \newcommand{\nf}{\mathsf{nf}}

 \usepackage{bbm}

\newcommand{\conn}{{\copyright}}
\newcommand{\conw}{\mathsf{conn}}

\theoremstyle{definition}
\newtheorem{definition}{\vspace{1mm}Definition}[section]

\theoremstyle{plain}

\newtheorem{theorem}[definition]{\vspace{1mm}Theorem}

\newtheorem{proposition}[definition]{\vspace{1mm}Proposition}

\newtheorem{example}[definition]{\vspace{1mm}Example}

\newcommand{\der}    						  {\vartriangleright}

\newcommand{\tuple}[1]                         {{\langle #1\rangle}}

\DeclareMathOperator*{\var}{\mathsf{var}}

\newcommand{\ou}          {\vee}
\newcommand{\e}          {\wedge}

\def\mxctsym#1{\hbox{\hbox to 0pt{$#1\land$}$#1\lor$}}
\newcommand*{\mxct}{\mathrel{\mathchoice{\mxctsym\displaystyle}
  {\mxctsym\textstyle}
  {\mxctsym\scriptstyle}
  {\mxctsym\scriptscriptstyle}}}

\newcommand{\pl}   {{\e\!\!\!\!\ou}}
\newcommand{\pll}          {\mxct}

\DeclareMathOperator*{\npl}{\mathsf{n}_{\pll}}

\newcommand{\ignore}[1]{}

 \date{ }

\title{
{
An  unexpected 
Boolean connective
}
}

\author{
S\'ergio Marcelino\\
{\tt 
smarcel@math.tecnico.ulisboa.pt} \\
{SQIG - Instituto de Telecomunica\c c\~oes}\\
{Dep. Matem\'atica - Instituto Superior T\'ecnico}\\
{Universidade de Lisboa, Portugal}}

\begin{document}

\maketitle

%
\begin{abstract}
We consider a 2-valued non-deterministic connective $\pl$ defined by the 
 table resulting from the entry-wise union of the tables of conjunction and disjunction.
 Being half conjunction and half disjunction we named it \emph{platypus}. 
 The value of $\pl$ is not completely determined by the input, contrasting with usual notion of Boolean connective.
We call non-deterministic Boolean connective any connective based on multi-functions over the Boolean set.
In this way, non-determinism  allows for an extended notion of
truth-functional connective.
%
%
Unexpectedly, this very simple connective and the logic it defines, 
illustrate various 
key advantages in working with generalized notions of semantics (by incorporating non-determinism), calculi (by allowing multiple-conclusion rules) and even of logic (moving from Tarskian to Scottian consequence relations).
%
We show that the associated logic 
cannot be characterized by any finite set of finite matrices, whereas with non-determinism two values suffice.
Furthermore, this logic is not finitely axiomatizable using single-conclusion rules, however 
%
we provide a very simple analytical multiple-conclusion axiomatization using only two rules.
Finally, 
deciding the associated multiple-conclusion logic is $\mathbf{coNP}$-complete, but deciding its single-conclusion fragment 
is in $\mathbf{P}$.

\end{abstract} 



\epigraph{ 
{So they cut him to pieces, wrote a thesis\\
A cranium of deceit, he's prone to lie and cheat;\\
 It's no wonder -- a blunder from down under\\
 Duckbill, watermole, duckmole!}}
 {\textup{Mr. Bungle}, Platypus~~~~~~~~~~~}

\section{Introduction}


This paper works as an overview of a series of concepts, results and techniques 
that have been 
yielding new insights in the analysis (and synthesis) of logics
%
%
in recent years. 
The power of these methods is illustrated by establishing various uncommon properties of a very simple non-deterministic  Boolean connective that we name \emph{platypus}. 
There are two crucial ingredients to be explored. 

On one hand, we depart from the traditional approach in logic of using (deterministic) semantics based on logical matrices, as proposed long ago by {\L}ukasiewicz and followers, 
and adopt a generalization of the standard logic matrix semantics proposed in the beginning of this century by  Avron and his collaborators
\cite{avlev05}, which allows the interpretation of the connectives to be based in multi-functions instead of simply functions.
The central idea is that in a non-deterministic matrix (Nmatrix) a connective can non-deterministically pick from a set of possible values instead of its value being completely determined by the input values.
This allows us to mix conjunction ($\land$) and disjunction ($\lor$) into a single connective ($\pl$)
such that
 for each input it
 may choose from the values 
output by conjunction and disjunction with that same input, enlarging the Boolean world beyond truth-functionality. 
This is reminiscent of the \emph{
platypus} 
whose appearance mixes bird and mammal traits and has generously lent its name to $\pl$.
Introducing the possibility of non-determinism has very powerful consequences. 
The most immediate advantage is that 
Nmatrices can finitely characterize logics that are not characterizable by (deterministic) matrices \cite{Avron,finval}.
As we shall see, this is also true for the logic of $\pl$.
The extra expressivity offered by non-determinism has also proven extremely valuable in obtaining recent compositional results in logic.
Namely, in
producing simple modular semantics for combined logics \cite{wollic}.
In particular yielding  finiteness-preserving semantics of strengthenings of a given logic with a set of axioms \cite{CaMaAXS}, covering a myriad of examples in the literature and explaining the emergency of structures like twist-structures \cite{
Oditwist,Umtwist} or swap-like structures \cite{coniglioswap}.
Further, non-determinism can also be used to give simple infectious semantics to a range of syntax-based relevant companions of a given many-valued logic \cite{ISMVL}.
%
%
 
 On the other hand, we consider 
 the 
 symmetrical multiple-premises{/}multiple-conclusions notion of logic introduced by Scott \cite{Scott}, and also Shoesmith and Smiley \cite{ShoesmithSmiley}, in the 1970s.  
This bilateralist view \cite{CCALJM}, generalizes the asymmetrical multiple-premises{/}single-conclusion approach of most modern logic, introduced as a mathematical object by Tarski and his followers. 
The gain in symmetry of expressive-power supports the effective development of analytic calculi for logics that could not even be finitely axiomatized before
 \cite{ShoesmithSmiley,synt,wollic19}.
The internalization of case analysis in the derivation mechanism,
yields nice proof-theoretical properties impossible in the single-conclusion setting.
These advantages have been mostly neglected in logic itself, but have been well-appreciated, for instance, in providing constructive proofs in algebra \cite{RINALDI2018226}. 
A key aspect of multiple-conclusion consequence is that it can be used to study the single-conclusion fragments of logics.
The logic of $\pl$ also allows us to illustrate such advantages.
Although it 
is not single-conclusion finitely axiomatizable, 
using the results in \cite{synt}
we were not only able to axiomatize it in an automated way, but also guarantee that the obtained axiomatization 
is analytic, a crucial property from the proof-theoretical point of view.
In \cite{synt} we also have shown how to get purely symbolic decision and proof-search procedures from analytical axiomatizations.
%
%
 From the compositionality point of view there are clear advantages in considering such calculi, involving no additional meta-language as in the (less pure) usual alternatives: sequent calculi, labeled tableaux or natural deduction. 
 This internal view of logic is also directly associated with the fundamental notion of logic as a consequence operation. 
 Notably, merging calculi for given logics 
 precisely captures the mechanism for combining logics known as fibring \cite{GabFib99,acs:css:ccal:98a}, yielding the least logic on the joint language that extends them.
 Such perspective allows us to better isolate, study and tame the origins of interactions in combined logics.
We detail the structure of the paper highlighting the most relevant results in each section.
In Section~\ref{sec:birth} we introduce the platypus connective:
first as a multi-function and show that together with any basis for the clone of all Boolean functions it forms a basis for the clone of all Boolean multi-functions (Prop.~\ref{decomp});
 then, as a non-deterministic logical connective characterized by a finite Nmatrix and show that this would not be possible using a finite matrix (Thm.~\ref{nofinmat}).
In Section~\ref{sec:ax} we explore the divide between single- and multiple-conclusion settings in terms of axiomatizability.
We show that platypus' logic is not finitely axiomatizable by a finite set of single-conclusion rules (Thm.~\ref{notfinax}) and provide an analytical axiomatization using two multiple-conclusion rules 
(Thm.~\ref{finax}).
In Section~\ref{sec:comp} we show that deciding the single-conclusion fragment of platypus' logic is in {\bf P} whilst seen as a multiple-conclusion logic it is {\bf coNP}-complete  (Thm.\ref{complexity}).
We wrap up in with Section~\ref{sec:final}, where we summarize the obtained results and open some doors for future work.


\section{The birth of platypus}\label{sec:birth}

The study of classical propositional logic is deeply connected to the study of Boolean functions, that is, maps $f:\mathbf{2}^n\to \mathbf{2}$ with $\mathbf{2}=\{0,1\}$.
The pioneering work on the clones of Boolean functions  by Post \cite{Post41} provides a complete analysis of the semantical expressivity of every fragment of classical logic. 
What happens if we consider functions that may output multiple values?

%
%
%

\subsection{From functions to multi-functions}

Multi-functions associate to each given input possibly more than one value.
An $n$-ary Boolean multi-function\footnote{In this paper exclude the possibility of a multi-function outputting the empty set however there are situations where this option is desirable
as we mention in the end of Section~\ref{sec:ax}.} is a map $f:\mathbf{2}^n\to \wp(\mathbf{2})\setminus \emptyset=\{\{0\},\{1\},\{0,1\}\}$. 
Of course, if for every input $\vec{x}\in \mathbf{2}^n$ we have that $f(\vec{x})$ is a singleton, then $f$ is also a (Boolean) function.
Observe that there
are $2^{2^n}$ $n$-ary Boolean functions and
 $3^{2^n}$ $n$-ary Boolean multi-functions. Hence, there are $81$ binary Boolean multi-functions, $16$ of them are Boolean functions, we will give particular 
 attention to one of the $65$ that are not functions.
%
The main actor in this paper, \emph{platypus}, can be seen as
a binary Boolean multi-function given by $\pl(x,y)=\{x\land y,x\lor y\}$.
Multi-functions can be easily represented as tables just like functions, check $\pl$ in tabular form in Example~\ref{ex:semantics}.
%
%
The multi-function $\pl$ was considered in \cite{Avron} in order to approximate the behaviour of a faulty AND gate, which responds correctly if the inputs are similar, and unpredictably otherwise. In this same sense $\pl$ can also be seen as a faulty OR gate.
\smallskip

A \emph{clone} 
 is a set of 
 functions, over some fixed set, closed by composition and containing every projection function ($\pi^n_i(x_1,\ldots,x_n)=x_i$ for every $n\in \nats$ and $1\leq i\leq n$). 
A set of 
functions \emph{generates} (or is a \emph{basis} for) a clone if this clone is exactly the smallest that contains that given set. 
This concept can be generalized to multi-functions once we fix a notion of composition.
 %
A natural possibility would be to see multi-functions as particular cases of relations.
In that case, we have that, when composing given $n$-ary $f$ and 
 $g_i$ for $1\leq i\leq n$, 
 for each input, the \emph{possible values} output by the $g_i$'s are accumulated through the composition, 
yielding that
%
$$
f\circ(g_1,\ldots,g_n)(\vec{x}_1,\ldots,\vec{x}_n)=f(g_1(\vec{x}_1),\ldots,g_n(\vec{x}_n))=
\bigcup\{f(\vec{y}):\vec{y}\in\!\! \prod\limits_{1\leq i\leq  n}g_i(\vec{x}_i)\}
$$


 As an alternative we may require that two terms representing the same formula 
 must have the same value when being fed to the outer function and obtain
$$f(g_1(\vec{x}_1),\ldots,g_n(\vec{x}_n))=
\bigcup\{f(\vec{y}):\vec{y}\in\!\! \prod\limits_{1\leq i\leq  n}g_i(\vec{x}_i),y_i=y_j \text{ if }g_i(\vec{x}_i)=g_j(\vec{x}_j)\}$$

Note that whenever $g_i$ for $1\leq i\leq n$ are functions both notions coincide. However, 
 if $g(0)=g(1)=\{0,1\}$, $f(0,0)=f(1,1)=\{0\}$ and $f(0,1)=f(1,0)=\{1\}$ then the composition yields
$f(g(x),g(x))=\{0,1\}$ with the first (more liberal) option and simply 
 $f(g(x),g(x))=\{0\}$ with the second. 
%
%
%
%

The first interesting property of platypus is that,
 together with any 
basis for the clone of all Boolean functions, it  forms a basis for the clone of all Boolean multi-functions.
Moreover, this is the case regardless of the notion of composition between the two we mentioned above.

\begin{proposition}\label{decomp}
 Given any Boolean multi-function $f$, there are Boolean functions $g_0$ and $g_1$ such that $f=\pl(g_0,g_1)$.
\end{proposition}
\begin{proof}
For each input $\vec{x}\in \mathbf{2}^n$ we have that either $f(\vec{x})=\{0\}$, $f(\vec{x})=\{1\}$ or $f(\vec{x})=\{0,1\}$.
Easily, the result follows by letting for $i=0,1$
 $$g_i(\vec{x})=
\begin{cases}
 \{i\} &\text{ if }f(\vec{x})=\{0,1\}\\
 f(\vec{x}) &\text{otherwise}\\
\end{cases}
$$
As $g_0$ and $g_1$ are functions the result of the composition is the same with both notions of composition we considered above.
\end{proof}

This decomposition is not unique and the following equalities hold:
 $$\pl(x,y)=\pl(y,x)=\pl(\pi_1^2(x,y),\pi_2^2(x,y))=\pl(x\land y,x\lor y)=\{x\land y,x\lor y\}=\{x,y\}.$$ 
%
%
Furthermore, if we consider the presentation of $\pl$ given by the last equality, and extend $\pl$ to act over a finite set $I$ as $x\pl_I y=\{x,y\}$ and we
add it to a basis for the clone of all functions over $I$, we obtain a basis for the clone of all multi-functions over $I$.
 That is, we can extend Proposition~\ref{decomp} to  
 the statement that for every multi-function $f$ over a fixed set $I$ of size $n$, we have that there are functions over $I$, $g_1,\ldots g_{n}$ such that
      $f=\pl(g_1, \pl(g_2, \pl(\ldots ,\pl(g_{n-1},g_n))))$. 
 One might argue that a good
alternative name\footnote{I thank Carlos Caleiro for suggesting \emph{platypus} and thus steering me away from  using such boring alternatives.} for $\pl_I$ 
could be (non-deterministic) \emph{union} or \emph{choice} (over $I$).

%
%
%
%

\subsection{Platypus as a non-deterministic  Boolean matrix}


For the sake of readability, we start by quickly revisiting some of the basic concepts needed, for more details see \cite{ShoesmithSmiley,Woj,Humb,font}.
A propositional signature 
is an indexed family $\Sigma=\{\Sigma^{(k)}:k\in \nats\}$ where $\Sigma^{(k)}$ are the k-ary connectives.
Given a set $X$, we denote by $L_\Sigma(X)$ the set of formulas  written with connectives in $\Sigma$ from the elements of $X$.
We consider fixed a (denumerable) set of propositional variables $P$. The propositional language associated to $\Sigma$ is 
$L_\Sigma(P)$. Along the paper,
whilst considering languages over signatures named $\Sigma_x$ 
we shall denote
$L_{\Sigma_x}
(P)$ simply by $L_x$.
Fixed a propositional language $L$, a (Scottian) \emph{logic} is a 
$\der\subseteq \wp(L)\times \wp(L)$ satisfying:
\begin{itemize}
\item[(\bf{O})] $\Gamma\der \Delta$ if $\Gamma\cap \Delta\neq \emptyset$ (\emph{overlap})
\item[(\bf{D})] $\Gamma\cup \Gamma' \der \Delta\cup \Delta'$ if $\Gamma\der \Delta$ 
(\emph{dilution})
\item[(\bf{C})] $\Gamma\der \Delta$ if $\Gamma\cup\Omega\der \overline{\Omega}\cup\Delta'$ for every partition $\tuple{\Omega,\overline{\Omega}}$ of $L$ 
(\emph{cut})
\item[(\bf{S})] $\Gamma^\sigma\der \Delta^\sigma$ for any substitution $\sigma:P\to L$ if $\Gamma\der \Delta$ (\emph{substitution invariance})
\end{itemize}

We say that $\der$ is \emph{finitary} whenever  $\Gamma\der \Delta$ implies  $\Gamma'\der \Delta'$ for 
 finite $\Gamma'\subseteq \Gamma$ and $\Delta'\subseteq \Delta$.
It is well known that this is a generalization of the notion of Tarskian logic. Indeed, given a Scottian logic $\der$, its single conclusion fragment
$\vdash_\der{=}\der{\cap} (\wp (L) \times L)$ is a Tarskian consequence relation satisfying:
\begin{itemize}
\item[(\bf{R})] 
$\Gamma\vdash \varphi$ if $\varphi\in \Gamma$ (\emph{reflexivity}), 
\item[(\bf{M})] 
$\Gamma\cup \Gamma' \vdash \varphi$ if $\Gamma\vdash \varphi$ 
(\emph{monotonicity}),
\item[(\bf{T})] 
$\Gamma\vdash \varphi$ if $\Delta\vdash \varphi $ and $\Gamma\vdash \psi$ for every $\psi\in \Delta$  
(
\emph{transitivity}) 
\item[(\bf{S})]
$\Gamma^\sigma\vdash \varphi^\sigma$ for any substitution $\sigma:P\to L$ if $\Gamma\vdash \varphi$ (\emph{substitution invariance})
\end{itemize}
Furthermore, $\vdash$ is \emph{finitary} whenever $\Gamma\vdash \varphi$ then $\Gamma'\vdash \varphi$
for some
finite $\Gamma'\subseteq \Gamma$.

We say that $\der$ is a multiple-conclusion companion of $\vdash_\der$.
Each Tarskian consequence relation $\vdash$ has potentially infinite multiple-conclusion companions.
The smallest among these, denoted by $\der_\vdash$, is defined as $\Gamma\der_\vdash \Delta$ if and only if 
there is $\varphi\in \Delta$ such that $\Gamma\vdash \varphi$. 
For example, $\vdash$ may be finitary but have non finitary companions, 
 but surely $\der_\vdash$ is finitary.
In an abstract sense any companion of $\vdash$ can be used to study $\vdash$, however there are advantages in considering companions that are particularly well behaved.
As we shall see, although  $\der_\vdash$ is a faithful representation of $\vdash$ in the multiple-conclusion setting, there are nicer ones that can be considered.
\smallskip
%
%

On an orthogonal direction, we will work also 
with a generalized notion of logical matrix.
Along with multi-functions comes an 
ingenious 
 extension of the standard notion 
 of matrix semantics introduced in \cite{avlev05}. 
A $\Sigma$-\emph{Nmatrix} is a tuple $\Mt=\tuple{V,\cdot_\Mt,D}$ where $V$ is the set of \emph{truth-values} and $D\subseteq V$ the set of \emph{designated} values.
Further, 
for each $\conn\in \Sigma^{(n)}$, $\cdot_\Mt$
yields 
$\conn_\Mt:V\to \wp(V)\setminus \emptyset$, \emph{interpreting} $\conn$ as a multi-function over $V$ instead of a function as in the case of matrices.
When complex formulas are interpreted over 
Nmatrices, the value is not completely determined by the values of the subformulas, instead, the multi-functions giving the interpretation of each connective
are read non-deterministically, 
allowing the valuation to choose for each formula 
 a different possible value. 
An $\Mt$-\emph{valuation} $v$ is a function $v:L_\Sigma(P)\to V$ satisfying $v(\conn(\varphi_1,\ldots,\varphi_k))\in \conn_\Mt(v(\varphi_1),\ldots,v(\varphi_k))$ for any $k$-place connective $\conn\in \Sigma$.
Just like 
in the usual matrix semantics, 
the logic characterized by an $\Sigma$-Nmatrix $\Mt$, $\der_\Mt$, is defined by $\Gamma\der_\Mt \Delta$ whenever
$v(\Gamma)\subseteq D$ implies $v(\Delta)\cap D\neq \emptyset$ for every $\Mt$-valuation $v$.  
As intended, the set of formulas in the left is read conjunctively and the one in the right disjunctively.
 We denote simply by 
 $\vdash_\Mt$ (instead of $\vdash_{\der_\Mt}$) the single-conclusion fragment of $\der_\Mt$, corresponding to the Taskian consequence relation defined by $\Mt$.

The usual logical matrix semantics is recovered when one considers
an Nmatrix for which every connective is interepreted as a function. 
Furthermore, Nmatrix semantics preserves various fundamental properties of matrix semantics.
%
%
 Every partial valuation defined over a set closed for subformulas can be extended to a full valuation (\emph{analyticity}).
Furthermore, for finite Nmatrix $\Mt$, $\der_\Mt$ and $\vdash_\Mt$ are \emph{finitary},
and deciding $\der_\Mt$ and $\vdash_\Mt$ is in {\bf coNP} (see \cite{Avron,synt,wollic19}).

\smallskip

\begin{example}\label{ex:semantics}\em
For $\conw\subseteq \{\land,\lor,\pl\}$, let $\Sigma_\conw$ contain exactly the binary connectives in $\conw$,
and $\BMt_\conw=\tuple{{\bf 2},\cdot,\{1\}}$ be the $\Sigma_\conw$-Nmatrix where for each $\conn\in \conw$
its interpretation multi-function 
is described in tabular form as
 \begin{center}
  \begin{tabular}{c | c c c }
$\land$ & $0$ & $1$  \\
\hline
$0$&  $ 0 $ & $ 0 $ \\
$1$ &$0$ & $1$ 
\end{tabular}\quad $\hookrightarrow$
  \begin{tabular}{c | c c c }
$\pl$ & $0$ & $1$  \\
\hline
$0$&  $ 0 $ & $ 0,1 $ \\
$1$ &$0,1$ & $1$ 
\end{tabular}\quad $\hookleftarrow$
  \begin{tabular}{c | c c c }
$\vee$ & $0$ & $1$  \\
\hline
$0$&  $ 0 $ & $ 1 $ \\
$1$ &$1$ & $1$ 
\end{tabular}
\end{center}
%
%
For ease of notation we do not distinguish between the connective and its interpretation, 
and write $\der_\conw$ instead of $\der_{\BMt_\conw}$,
and 
$\BMt_{\pll}$ and $\BMt_{\land\lor\pl}$ instead of $\BMt_{\{\pll\}}$ and $\BMt_{\{\land,\lor,\pll\}}$. \hfill$\triangle$
\end{example}

Whenever
 $\pl\notin \conw$ we have that $\BMt_\conw$ is the well known Boolean matrix characterizing classical logic in the corresponding signature.
For now, let us look closer at platypus' logic $\der_{\pll}{=}\der_{\BMt_{\pll}}$ and its single-conclusion fragment $\vdash_{\pll}{=}\vdash_{\der_{\pll}}{=}\vdash_{\BMt_{\pll}}$.

As the inclusion arrows indicate, the interpretation of $\land$ and $\lor$ are contained in the interpretation of $\pl$. 
Meaning that if we use the tables of $\land$ or $\lor$ to interpret $\pl$ we obtain a strict subset of valuations.
Yielding,
$$\der_{\pll}\subseteq t_{\pll}(\der_\land) \cap t_{\pll}(\der_\lor)$$
where $t_{\pll}$ swaps every occurrence of $\land$ and $\lor$ by $\pl$.
This fact has an interesting immediate consequence: $\der_{\pll}$ inherits the property of being a right-inclusion logic from $\land$, and of being a left-inclusion logic from $\lor$.
We say that a logic\footnote{This notion is usually presented over Tarskian consequence relations. 
In such setting, due to the asymmetry in the very notion of logic, the two notions are not symmetric. 
We will not enter in details here but just state that for \emph{positive} logics $\der$ where $L_\Sigma(P)\not\der \emptyset$ we have that
$\der$ is left-(right-)inclusion logic iff $\vdash_\der$ is. We point to \cite{Bonzio2018_leftinclusion,ISMVL} for more details.
  }
 $\der$ is of \emph{left-inclusion} (\emph{right-inclusion}) whenever $\Gamma\der \Delta$ implies there are $\Gamma'\subseteq \Gamma$ and $\Delta'\subseteq \Delta$ such that  $\var(\Gamma)\subseteq \var(\Delta)$ ($\var(\Gamma)\supseteq \var(\Delta)$)
where $\var(\varphi)$ denotes the set of variables occurring in $\varphi$. 


%

\begin{proposition}\em
If $\Gamma\der_{\pll} \Delta$ then $\var(\Gamma')=\var(\Delta')$ for some 
 %
 $\Gamma'\subseteq \Gamma$, $\Delta'\subseteq \Delta$. 
\end{proposition}
\begin{proof}
It follows immediately from the previous observation and the fact that  $\var(t_{\pll}(\Gamma))=\var(\varphi)$ for every $\varphi\in L_\land\cup L_\lor$.
\end{proof}

\smallskip


 It is well know that every logic $\der$ is characterized by some family of matrices.
 For example, by the family of 
 matrices $\Mt$ satisfying 
 $\der {\subseteq} \der_\Mt$. 
 Or even by of subset of these, the Lindenbaum bundle of $\der$ formed by $\Mt_\Gamma{=}\tuple{L_\Sigma(P),\cdot,\Gamma}$ for
 $\Gamma\subseteq L_\Sigma(P)$ satisfying $\Gamma \not\der \overline{\Gamma}$ with $\overline{\Gamma} {=} L_\Sigma(P)\setminus \Gamma$ 
 (see \cite{ShoesmithSmiley,font,CCALJM}).
 The same is true of course in the single-conclusion context.
 %
 However can $\der_{\pll}$ or $\vdash_{\pll}$ be characterized by a single matrix? 
 How could such semantics look like?

\begin{proposition}\em
There is a (deterministic) $\Sigma_{\pll}$-matrix $\Mt$ such that $\der_{\pll}=\der_\Mt$.
\end{proposition}
\begin{proof}
   From \cite{ShoesmithSmiley}[Thm. 15.2] we know that a logic $\der$ is definable by a single matrix if and only if it permits \emph{cancellation}. That is, 
if 
$\bigcup_k{X_k}\der \bigcup_k{Y_k}$ and 
$\var(X_i\cup Y_i)\cap \var(X_i\cup Y_j)$ for $i\neq j$, then 
there is $i$ such that $X_i\der Y_j$.
It is easy to see that $\der_\BMt$ permits cancellation. Given $\BMt$-valuations $v_i$ such that $v_i(X_i)=\{1\}$ and $v_i(Y_i)=\{0\}$, we consider the partial valuation defined over
$Z=\bigcup_k L_{\Sigma_{\pll}}(\var(X_k\cup Y_k))$ making 
$v(\varphi)=v_k(\varphi)$ for $\varphi\in  L_{\Sigma_{\pll}}(\var(X_k\cup Y_k)$.
As $Z$ is closed under taking subformulas it can be extended to a full $\BMt$-valuation $v$ showing that $\bigcup_k{X_k}\not\der_\BMt \bigcup_k{Y_k}$.
\end{proof}

Therefore, $\der_{\pll}$, and hence also $\vdash_{\pll}$, can be characterized by a single matrix, however, as we will see, it cannot be a finite one.
Given $\varphi\in L_{\pll}$ let $\nf(\varphi)$ be the smallest subformula $\psi$ of $\varphi$ such that $\varphi\in L_{\pll}(\{\psi\})$.
That is, $\nf((p\pl p)\pl p)=\nf(p\pl p)=\nf(p)=p$ and
$\nf((p\pl q)\pl(p\pl q))=\nf(p\pl q)=p\pl q$.
Let also
\begin{align*}
 \pl^{0}(\varphi)&=\varphi\\
 \pl^{n+1}(\varphi)&=(\pl^n (\varphi)) \pl \varphi
\end{align*}
%
%
It is straightforward to see that for every formula $\varphi$, $\varphi\dashv\vdash_{\pll} \nf(\varphi)$.
In particular $\varphi\dashv\vdash_{\pll} p=\nf(\varphi)$ for every $\varphi\in L_{\pll}(\{p\})$. 
Every pair of formulas written in one variable are equivalent, however when two variables are present the situation dramatically changes and allows us to show the following result.

\begin{theorem}\label{nofinmat}\em
There is no finite matrix $\Mt$ such that
 $\vdash_\BMt=\vdash_\Mt$ 
or $\der_\BMt=\der_\Mt$. 
\end{theorem}
\begin{proof}
 For each natural $n$, let $\varphi_n=(\pl^n p) \pl q$. 
 Any $\BMt$-valuation such
 that $v(\pl^n p)=1$ for every $n\in \nats$, $v(q)=0$, $v((\pl^i p) \pl q)=1$ and $v((\pl^j p) \pl q)=0$ for  $i\neq j\in \nats$, showing that
 $\varphi_i\not\vdash_\BMt \varphi_j$.
Hence, $\dashv\vdash_\BMt$ splits $L_{\Sigma_{\pll}}(\{p,q\})$ in an infinite number of equivalence classes.
Hence,  
$\vdash_\BMt$ is not locally tabular and hence by \cite{finval} the first part of the result stands.
The second part follows immediately as any $\Mt$ such satisfying $\der_\BMt{=}\der_\Mt$ satisfies also $\vdash_\BMt{=}\vdash_\Mt$.
%
\end{proof}




%

It is easy to see that $p \der_{\pll} p\pl p$ and $p \pl p  \der_{\pll}  p$ but $(p\pl p)\pl q \not\der_{\pll}  p\pl q$
as is shown in the previous proposition. Hence, $\der_{\pll}$ is not self-extensional as it lacks substitution by equivalents.


\subsection{Abbreviations over Nmatrices}
As we have been seeing, and will be ever more clear along the remaining of this paper,
$\pl$ behaves in many different ways from its deterministic Boolean cousins. 
In fact, non-determinism opens the door to various unexpected phenomena. 
A notable difference regards considering connectives defined by abbreviation.
As in the deterministic case, given an $\Sigma$-Nmatrix each formula $\varphi\in L_{\Sigma}(P)$ (in $n$ variables),
defines an 
$n$-ary multi-function 
$$\varphi_\Mt(x_1,\ldots,x_n)=\{v(\varphi(p_1,\ldots,p_n)):v \text{ is } \Mt\text{-valuation}, v(p_i)=x_i\text{ for }1\leq i\leq n \}$$
Note that $\varphi_\Mt$ does not correspond to the composition of the 
interpretation of the connectives (as multi-functions) forming $\varphi$ using any of the notions we mentioned in the previous subsection.
Although it is closer to the second, it is still more restrictive. 
Given $f$ defined by $\varphi$ and $g_i$ by $\psi_i$, their 
composition 
is given by $f\circ(g_1,\ldots,g_n)=\varphi(\psi_1,\ldots,\psi_n)_\Mt$.
Whenever composing $f(g_1(h(x)),g_2(h(x)))$ we must guarantee that the values fed to $f$,
must come from values of $g_1(h(x))$ and $g_2(h(x))$ for the same values of $h(x)$.
Nonetheless, Prop.~\ref{decomp} still applies to this notion of composition of (expressible) multi-functions.
%
%
Another particularity of definitions by abbreviation in the context of Nmatrices is that 
whenever two formulas determine the same function (instead of a multi-function) we have that they are logically interchangeable in every context.
However when 
some input may output more than a value this is not necessarily the case.
For example, $\pl$ is given by a symmetric table, and indeed $\pl(x,y)$ and $\pl(y,x)$ define the same multi-function, we have that
$p\pl q\not\der_{\pll} q\pl p$.
This means that when working with Nmatrices we have to be careful about what we expect from 
a connective defined by abbreviation.
Clearly, it may lose any connection with the connectives used to define it,
as the possibility of independent choices offered by the non-determinism may brake relation between them in the resulting logic. 
This is a question that must be taken into acount 
by any approach to the logics 
of Nmatrices using clones of multi-functions. 
\smallskip

\section{Axiomatizability}\label{sec:ax}

Both  Tarskian and Scottian logics are associated with 
very natural notions of axiomatizability according 
to their type. 
A set of sound rules $R \subseteq {\vdash}$
\emph{axiomatizes} (or is a \emph{basis} for) $\vdash$ whenever $\vdash$ is 
the closure of $R$ under 
({\bf R})({\bf M})({\bf T})({\bf S})
 in which case we write $\vdash_R{=}\vdash$.
 Analogously, 
 set \emph{sound} rules $R \subseteq {\der}$ 
\emph{axiomatizes} (or is a \emph{basis} for) 
$\der$ 
whenever $\der$ is 
the closure of $R$ under 
({\bf O})({\bf D})({\bf C})({\bf S}), and in that case we write $\der_R{=}\der$. 
These definitions fare well on the compositional front, as
given two logics of compatible type axiomatized by sets of rules $R_1$ and $R_2$ (of according type)
their \emph{fibring}, the smallest logic (of the same type) in the combined language that contains both,
is axiomatized
 by $R_1\cup R_2$.


Crucially, in both cases the abstract properties defining each type of calculi correspond to the machinery of Hilbert-style calculi where derived consequences  using a set of rules $R$ are exactly the ones that hold in the logic axiomatized by $R$.
In the single-conclusion case, derivations are sequences where the application of a rule produces a new formula,
and in the multiple-conclusion case the proofs take an arboreal shape since the application of the rules produces set of formulas, each corresponding to a child of the node where it was
 applied\footnote{Rules with empty set of conclusion discontinue the branch of the node where it is applied. 
 We have that $\Gamma\der_R \Delta$ whenever there is a $R$-derivation departing from $\Gamma$ where the leaf of each non discontinued branch must be a formula in $\Delta$.}.
The second notion strictly generalizes the first as derivations using only single-conclusion rules coincide in both settings.
If $R$ is a set of single conclusion rules $\der_R=\der_{\vdash_R}$.
%
For a formal definitions and illustrate examples of such derivations we point to \cite{ShoesmithSmiley,synt,wollic19}. 

%

\subsection{
Single-conclusion rules only}

Rautenberg has shown in \cite{Raut} that every fragment of classical logic is finitely axiomatizable (using single-conclusion rules).

\begin{example}\label{ex:axiomatizationssc}\em
Consider the following well known axiomatizations for the logics of the deterministic reducts of $\BMt_{\land\lor\pl}$.
For $\pl \notin \conw$ we have that
$\vdash_{R_\conw}{=}\vdash_{\conw}$
 with 
\begin{align*}
R^{\scc}_{\land}&= \{\frac{p\land q}{p}\ _{r^\land_1}\,,\, \frac{p\land q}{ q}\ _{r^\land_2}\,,\, \frac{p \,,\, q}{p\land q}\ _{r^\land_3}\} \\
    R^{\scc}_\lor&=\{\frac{p}{\;p\ou q\;} \,,\,
    \frac{\;p\ou p\;}{p}  \,,\,
    \frac{\;p\ou q\;}{q\ou p}  \,,\,
    \frac{\;p\ou (q \ou r)\;}{(p \ou q)\ou r}\}\\
    R^{\scc}_{\land\lor}&=R^{\scc}_\land\cup R^{\scc}_\lor \cup \{ \frac{\;p\ou q\;\;\;p\ou r\;}{\;p\ou (q\e r)\;},\frac{\;p\ou (q\e r)\;}{\;p\ou q\;},\frac{\;p\ou (q\e r)\;}{\;p\ou r\;}\}
\end{align*}
 %
 %
 %
%
We have shown in \cite{soco} that the rules mixing $\land$ and $\lor$ are fundamental in the single-conclusion setting to capture the interaction between these connectives,
contrasting with what happens in the multiple-conclusion setting,
as we shall see latter on in this section.
\text{}\hfill$\triangle$
\end{example}

How about $\vdash_{\pll}$? Can Rautenberg's result be extended to Boolean Nmatrices? 
The next theorem shows that the answer is negative, but first a proposition giving an useful recursive characterization of $\vdash_{\pll}$.
%

\begin{proposition}\label{prop:scpl}\em
 $\Gamma\vdash_{\pll} \varphi$ if and only if $\varphi\in L_{\pll}(\nf(\Gamma))$.
\end{proposition}
\begin{proof}
 From right to left, let $v$ be an
 $\Mt_{\pll}$-valuation such that $v(\Gamma)=\{1\}$. 
  As
  $\psi\vdash_{\pll} \nf(\psi)$ for every $\psi\in\Gamma$ we get that $v(\nf(\Gamma))=\{1\}$, and
 from $ {p, q}\vdash_{\pll}{p\pl q}$ we get that $v(\psi)=1$ for $\psi\in L_{\pll}(\nf(\Gamma))$.  

 From left to right, we show that if $\varphi\notin L_{\pll}(\nf(\Gamma))$ then there is a $\BMt_{\pll}$-valuation such that $v(\Gamma)=\{1\}$ and $v(\varphi)=0$.
 By induction on the structure of $\varphi$.
  If $\varphi=p\in P$ then $v$ such that $v(\psi)=0$ iff $\psi=p$ 
   does the job, as every formula $\Gamma$ must contain 
 some variable different from $p$. 
 If $\varphi=\varphi_1\pl \varphi_2$ then there must be $i\in \{1,2\}$ such that $\varphi_i\notin L_{\pll}(\Gamma)$. 
 By induction hypothesis there is $v$ such that $v(\Gamma)=1$ and $v(\varphi_i)=0$. Hence, we can define some $v'$ that coincides with $v$ on the set of subformulas of $\Gamma$, 
and makes $v'(\varphi_1\pl \varphi_2)=0$.
\end{proof}

For $n\in \nats$, let 
\begin{align*}
 R_n&=\Bigl\{\frac{\varphi}{p}:\varphi\in L_\pl(\{p\}), \npl(\varphi)\leq n\Bigr\}\\
 R_\omega&=\bigcup_{i<\omega} R_i\\
 R^\scc_{\pll}&=\Bigl\{\frac{p\, , \, q}{p\pl q}\Bigr\}\cup R_\omega
\end{align*}
where $\npl(\varphi)$ is the number of occurrences of $\pl$ in $\varphi$.
%
%

\begin{theorem}\label{notfinax}\em
 $\vdash_{\pll}$ is axiomatized by $R^\scc_{\pll}$ and it is not finitely axiomatizable. 
\end{theorem}
\begin{proof}
That $\vdash_{\pll}{=}\vdash_{R^\scc_{\pll}}$ follows easily from Prop.~\ref{prop:scpl} by observing that the rules in $R_\omega$ are enough to produce $\nf(\Gamma)$ from $\Gamma$
 and
 $\frac{p\, , \, q}{p\pl q}$ is enough to generate every formula in $L_{\pll}(\nf(\Gamma))$ from  $\nf(\Gamma)$.


To see that $\vdash_{\pll}$ is not finitely axiomatizable,
it is enough to show that $\vdash_{\pll}$ is the limit of an infinite strictly increasing sequence $\vdash_n$ (see \cite[Thm.2.2.8]{Woj}). 
Consider the logic $\vdash^n$ axiomatized by 
 $\{\frac{p\, , \, q}{p \pl q}\}\cup R_n$ for each $n\in \nats$.
Clearly, $\vdash_n {\subseteq}\vdash_{n+1} {\subseteq} \vdash_{\pll}$ for each $n$ and $\vdash_{\pll}$ is the limit of this sequence.
It remains to show that $\vdash_n{\subsetneq} \vdash_{n+1}$.
 Let
 $\Mt_n=\tuple{\{a_0,\ldots,a_n,1\},\cdot_n,\{1\}}$ with
  %
\begin{align*}
  \pl_n(1,x)&=\pl_n(x,1)=\{1,x\}\\
  \pl_n(a_0,a_n)&=\pl_n(a_n,a_0)=\{1,a_0\}\\
 \pl_n(a_i,a_j)&=\bigl\{a_{max(0,min(i,j)-1)}\}: \{i,j\}\neq \{0,n\}\bigr\}
\end{align*}
It is easy to check that $p\,,\,q\vdash_{\Mt_n} p\pl q$ and $\pl^{k}(p)\vdash_{\Mt_n}p$ for $k\leq n$.
 To show that $\pl^{n+1}(p)\not\vdash_{\Mt_n}p$, 
 consider an $\Mt_n$-valuation $v$ such that
 $v(p)=a_n$, $v(p\pl p)=a_{n-1}$, \ldots, $v(\pl^n(p))=a_{0}$ and $v(\pl^{n+1}(p))=v(\pl^{n}(p)\pl p)=1$.
\end{proof}


\subsection{Allowing multiple-conclusion rules}

The fact that for certain finite Nmatrices $\Mt$, $\vdash_\Mt$ is non-finitely axiomatizable using single-conclusion rules is nothing that non-determinism can be blamed for.
In \cite{wronski74}, Wr\'onski shown that $\vdash_\CMt$ for the
 (deterministic) $\Sigma_\bullet$-matrix $\CMt=\{\{0,1,2\},\cdot_\CMt,\{2\}\}$
where $\Sigma_\bullet$ contains a single binary connective $\bullet$ and
\begin{center}
          \begin{tabular}{c | c c c}
   $\bullet_{\CMt}$ &  $0$ & $1$ & $2$   \\ 
    \hline
    $0$ & $1$ & $2$ & $2$ \\ 
     $1$  &$2$ & $2$ & $2$\\
     $2$  &$1$ & $2$ & $2$    
  \end{tabular}	
\end{center}
is not finitely (single-conclusion) axiomatizable.
However, $\der_\CMt$ is axiomatized multiple-conclusion by the following $6$ rules \cite{synt}:
$$\frac{ }{\;(p\bullet q)\bullet (p\bullet q)\; }\qquad \frac{q}{\;p\bullet q\;}\qquad \frac{q\bullet q}{\;q\,,\,p\bullet q\;}$$
$$\frac{p\bullet p}{\;p\,,\,p\bullet q\;}\qquad \frac{\;p\;,\;p\bullet q\;}{\,q\bullet q}\qquad \frac{\;p\bullet q\;}{\;p\bullet p\;,\;q\bullet q\;}$$

The first advantage of working in the multiple-conclusion setting is that every finite (deterministic) matrix is finitely axiomatizable \cite[Thm.19.12]{ShoesmithSmiley},
whereas in the single-conclusion this fails already for
matrices of size $3$.
%
%
%
%
%
%
%
%
%

In \cite{synt} we shown that this result could 
be extended to every finite 
Nmatrix provided it is \emph{monadic}, 
 a reasonable expressiveness requirement. 
An Nmatrix is $\Mt=\tuple{V,\cdot_\Mt,D}$ is \emph{monadic} if there is a set of formulas in one variable $S\subseteq L_{\Sigma}(\{p\})$ \emph{separating} $\Mt$, that is, such that for each every pair of distinct elements of $\Mt$ there is $\varphi\in S$ such that
$\varphi_\Mt(x)\subseteq D$ and $\varphi_\Mt(y)\subseteq V\setminus D$, or $\varphi_\Mt(x)\subseteq V\setminus D$ and $\varphi_\Mt(y)\subseteq D$.
In the $3$-valued matrix $\CMt$ this requirement is met with $S=\{p,p\bullet p\}$ which allowed us to produce the axiomatization above.
\smallskip

Moreover, it is easy to check that for any signature $\Sigma$ and Boolean Nmatrix $\Mt=\tuple{{\bf 2},\cdot_\Mt,\{1\}}$ we have that the set $S=\{p\}$ separates $\Mt$.
In this (Boolean) case, 
the general strategy to axiomatize $\der_\Mt$ introduced in \cite{synt}  
boils down to collecting the rules $R_{\conn,\vec{x}}$ for each $\conn\in \Sigma^{(n)}$ and $\vec{x}\in {\bf 2}^n$
 where $\conn_\Mt(\vec{x})\neq\{0,1\}$: 
\begin{itemize}
 \item[] if $\conn_\Mt(\vec{x})=\{0\}$ make
 $$R_{\conn,\vec{x}}=\frac{\{p_i:x_i=1\}\cup \{\conn(p_1,\ldots,p_k)\}}{\{p_i:x_i=0\}}$$ 
 
\item[] if $\conn_\Mt(\vec{x})=\{1\}$ make
 
 $$R_{\conn,\vec{x}}=\frac{\{p_i:x_i=1\}}{\{p_i:x_i=0\}\cup \{\conn(p_1,\ldots,p_k)\}}$$
 
\end{itemize}

Note that this axiomatization is completely modular on  the connectives being considered.
Furthermore, it is modular in the entries of the table defining each connective, 
a single rule is collected for each entry where $\conn_\Mt(\vec{x})$ is a singleton.
Yielding the following axiomatization of $\der_{\pll}$.

\begin{theorem}\label{finax}\em
$\der_{R_{\pll}}\,=\,\der_{\pll}$ with 
 $$R_{\pll}\,=\,\Bigl\{
 \frac{p\, , \, q}{p\pl q}\ _{r^\scc_{\pll}} \,,\, 
 \frac{p\pl q}{p\, , \, q}\ _{r^\mc_{\pll}}\Bigr\}
 $$
\end{theorem}
{
As a corollary we obtain that $\pl$ is the mysterious connective introduced in \cite[Exercise 4.31.5]{Humb}, by observing
 that a multiple-conclusion rule $\frac{\Gamma}{\Delta}$ is equivalent to the metarule
$\frac{\bigl\{\Gamma'\vdash A: A\in \Gamma\bigr\}\bigcup\bigl\{\Gamma_i,B_i\vdash C: B_i\in \Delta\bigr\}}{\Gamma'\,,\,\bigcup_i\Gamma_i\vdash C}$.
}

This recipe can be applied to the deterministic case too, yielding $\der_\land=\der_{R^\scc_\land}$ (so $\der_\land=\der_{\vdash_\land}$)
and $\der_\lor=\der_{R^\mc_\lor}$ with $R^\mc_\lor=\bigl\{\frac{p}{p\lor q},\frac{q}{p\lor q},\frac{p\lor q}{p\,,\,q}\bigr\}$.
Using the modularity on the connectives immediately we obtain that, for $\conw\in \{\land\pl,\lor\pl,{\land\!\lor\!\pl}\}$ we have $\der_{\conw}=\der_{R^\mc_{\conw}}$
 with $R^\mc_{\land\pl}=R^\scc_{\land}\cup R^\mc_{\pll}$,
$R^\mc_{\lor\pl}=R^\mc_{\lor}\cup R^\mc_{\pll}$ and $R^\mc_{\land\lor\pl}=R^\scc_{\land}\cup R^\mc_{\lor}\cup R^\mc_{\pll}$.

\begin{example}\label{ex:derivations}\em
 The following three multiple-conclusion derivations illustrate how multiple-conclusion derivations can be useful even if the set of conclusions is a singleton.
 
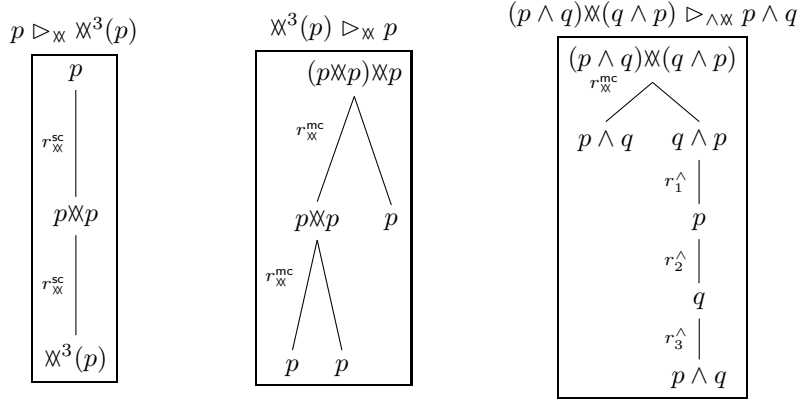
\begin{minipage}{.23\textwidth}
\begin{center}
$p\der_{\pll} \pl^3(p)$\\[.1cm]
\frame{\begin{tikzpicture}
\tikzset{level distance=55pt}
\tikzset{sibling distance=7pt}
 \Tree[.$p$ 
  \edge node[auto=right] {$_{r^\scc_{\pll}}$};
  [.$p\pl p$
   \edge node[auto=right] {$_{r^\scc_{\pll}}$};
   [.$\pl^3(p)$
      ]]    ] 
\end{tikzpicture}}
\end{center}
\end{minipage}\quad
\begin{minipage}{.28\textwidth}
\begin{center}
$\pl^3(p)\der_{\pll} p$\\[.1cm]
%
\frame{\begin{tikzpicture}
\tikzset{level distance=55pt}
\tikzset{sibling distance=7pt}
 \Tree[.$(p\pl p)\pl p$ 
  \edge node[auto=right] {$_{r^\mc_{\pll}}$};
  [.$p\pl p$
   \edge node[auto=right] {$_{r^\mc_{\pll}}$};
   [.$p$
   ] [.$p$
      ]]  [.$p$
    ] ] 
\end{tikzpicture}}
\end{center}
\end{minipage} 
\begin{minipage}{.4\textwidth}
\begin{center}
$(p\land q)\pl (q\land p)\der_{\land\pl} p\land q$\\[.1cm]
\frame{\begin{tikzpicture}
\tikzset{level distance=30pt}
\tikzset{sibling distance=5pt}
 \Tree[.$(p\land q)\pl (q\land p)$ 
 \edge node[auto=right] {$_{r^\mc_{\pll}}$};
  $p\land q$ 
  [.$q\land p$
   \edge node[auto=right] {$_{r^\land_1}$};
   [.$p$ 
    \edge node[auto=right] {$_{r^\land_2}$};
   [.$q$ 
    \edge node[auto=right] {$_{r^\land_3}$};
   [.$p\land q$ ]
   ]
   ]
   ] ] 
\end{tikzpicture}}
\end{center}
\end{minipage}\\

The first derivation examplifies how formulas in $L_{\pll}(\{p\})$ can be derived from $p$, using the rule $r^\scc_{\pll}$ shared by $R^\scc_{\pll}$ and $R^\mc_{\pll}$.
 In the second we show 
 how in this more expressive setting we can
 recover 
 the infinite rules in $R_\omega$ needed to axiomatize $\vdash_{\pll}$.
 In the second we show a valid rule in $\BMt_{\land\pl}$ that will be useful later on. 
 %
 \hfill$\triangle$
\end{example}

Note that there are Nmatrices that are finitely axiomatizable using only single-conclusion rules. 
For example, the above strategy applied to the Boolean $\Sigma_\to$-Nmatrix $\BMt_\to$ where $\Sigma_\to$ contains a single
 binary connective $\to$ interpreted by the table 
\begin{center}
  \begin{tabular}{c | c c c }
$\to$ & $0$ & $1$  \\
\hline
$0$&  $ 0 $ & $ 0,1 $ \\
$1$ &$0,1$ & $1$ 
\end{tabular}
\end{center}
%
yields the single rule 
of \emph{modus ponens} $R_{\mathsf{mp}}=\{\frac{p\,,\,p\to q}{q}\}$. Implying that $\der_\Mt=\der_{R_{\mathsf{mp}}}$ is the smallest companion of $\vdash_R$,
 $\der_{\vdash_R}$.
Note that $\vdash_{\BMt_\to}$ is also not characterizable by finite matrix \cite{finval,synt}.

There is still another advantage in considering multiple-conclusion rules.
The result in \cite{synt} extends the result in \cite{ShoesmithSmiley} in yet another way, a big novelty on Hilbert-style calculi (that avoid the inclusion of meta-language in the derivation mechanism).
The obtained calculi are analytical, in the sense that 
$\Gamma\der_R \Delta$ if and only if there is a $R$-derivation where only subformulas of $\Gamma$ and $\Delta$ appear.
This allows for effective purely symbolic decision procedures for the logic and 
 proof-search mechanism, see \cite{synt,wollic19}.
%
%
%
%
%
%

\subsection{Compositionality: single- vs multiple-conclusion} 

As we mentioned above there are great advantages on the compositionality front in avoiding any meta-language. 
The smallest  (single- or multiple- conclusion) logic that contains two (single- or multiple- conclusion) logics is axiomatized by joining axiomatizations of both logics.
This allows us to control the desired interactions. 
Such compositional mechanisms also  enlighten us regarding the divide single-/multiple- conclusion. 
In \cite{soco}, by taking profit from the complete characterization of the Boolean clones (of functions) given by Post \cite{Post41},  we have shown that fragments of 
classical logic seen as a single-conclusion logics cannot in general be axiomatized by joining the single-conclusion axiomatizations of each of the connectives in that fragment.
This sharply contrasts with what happens in multiple-conclusion, as is evident by the general recipe for axiomatizing Boolean Nmatrices presented above.
The fact is that the single-conclusion fragment of the fibring of two multiple-conclusion logics differs from the fibring of their single-conclusion fragments. 
As shown in Example~\ref{ex:derivations} $(p\land q)\pl (q\land p)\der_{R^\mc_{\land\pl}} p\land q$,
however by observing the rules in $R^\scc_{\pll}\cup R^\scc_\land$ one can sense that 
$(p\land q)\pl (q\land p)\not\vdash_{R^\scc_{\pll}\cup R^\scc_\land} p\land q$. How can we prove it? 

Recent work (some still unpublished) on fibring semantics shines a new light on this phenomena. We will now give a glimpse over it.
%
We say that an Nmatrix $\Mt=\tuple{V,\cdot_\Mt,D}$ is \emph{saturated} whenever for every $\vdash_\Mt$-theory $\Gamma$
 there is a $\Mt$-valuation $v$ such that that designates $v(\psi)\in D$ iff $\psi\in\Gamma$. 
 Equivalently, $\Mt$ is
 saturated iff $\Gamma\not\vdash_\Mt \psi$ for every $\psi\in\Delta$ then $\Gamma\not\der_\Mt \Delta$. 
 Given two Nmatrices $\Mt_1=\tuple{V_1,\cdot_1,D_1}$ and $\Mt_2=\tuple{V_2,\cdot_2,D_2}$ over disjoint signatures 
let $\Mt_1\star \Mt_2=\tuple{V_{12},\cdot_{12},D_{12}}$ where $V_{12}=(D_1\times D_2)\cup (V_1\setminus D_1\times V_2\setminus D_2)$,
$D_{12}=D_1\cup D_2$ and $$\conn_{12}(\vec{x})=\{(y_1,y_2)\in V_{12}: y_i\in\conn_i \text{ if }\conn\in \Sigma_i\}.$$
In \cite{wollic} we have show that given two Nmatrices $\Mt_1=\tuple{V_1,\cdot_1,D_1}$ and $\Mt_2=\tuple{V_2,\cdot_2,D_2}$ over disjoint signatures 
their fibring is simply characterized by $\Mt_1\star \Mt_2$ whenever
$\Mt_1$ and $\Mt_2$ are \emph{saturated}. 
For every $\Mt=\tuple{V,\cdot_\Mt,D}$ let
 $\Mt^\omega=\tuple{V^\omega,\cdot_\omega,D^\omega}$ where $\cdot_\omega$ interprets each connective component wise. 
As for any Nmatrix $\Mt$ we have that $\vdash_\Mt=\vdash_{\Mt^\omega}$ and $\Mt^\omega$ is
  always saturated 
we obtain a general semantics for 
joining single-conclusion rules:
if $\vdash_{\Mt_i}=\vdash_{R_i}$ for $i\in \{1,2\}$ then $\vdash_{R_1\cup R_2}=\vdash_{\Mt_1^\omega\star \Mt_2^\omega}$.

We know that $\BMt_\land$ is saturated, hence $\vdash_{R^\scc_{\pll}\cup R^\scc_\land}=\vdash_{\BMt_{\pll}^\omega\star \BMt_\land}$.
It is not hard to show that $\BMt^\omega_{\pll}\star \BMt_\lor$ is isomorphic\footnote{Using ${\bf 2}^\omega\approx \wp(\nats)\approx \{(\nats,1)\}\cup 
\{(0,X):X\subsetneq \nats\}$.}
 to $\tuple{\wp(\nats),\widetilde{\cdot},\{\nats\}}$ where
$$X\widetilde{\pl} 
Y=\{Z\in \wp(\nats):X\cap Y\subseteq Z\subseteq 
\nats\setminus(X\cup Y)\}$$
$$X\tilde{\land} 
Y=
\begin{cases}
 \{\nats\}& \text{ if }X=Y=\nats\\
  \wp(\nats)\setminus\{\nats\}& \text{ otherwise.}\\
\end{cases}
 $$
 Now we can finally show that indeed $(p\land q)\pl (q\land p)\not\vdash_{R^\scc_{\pll}\cup R^\scc_\land} p\land q$.
Clearly defining $v(p)=v(q)=v(p\land q)=v(q\land p)=\emptyset$ and 
$v((p\land q)\pl (q\land p))=\nats$ makes $v$ a partial $\BMt_{\pll}^\omega\star \BMt_\land$-valuation defined over a set closed for subformulas and hence the result follows. 
%
As a consequence we get that $\BMt_{\pll}$ cannot be saturated. Furthermore there is no finite
saturated  Nmatrix characterizing $\vdash_{\pll}$. 

\begin{proposition}\em
 There is no finite saturated Nmatrix $\Mt$ such that $\vdash_\Mt{=}\vdash_{\pll}$.
\end{proposition}
\begin{proof}
 Let $\Gamma_n=\{p_i\pl p_j: 0\leq i\neq j \leq n\}$ and $\Delta_n=\{p_i:0\leq i \leq n\}$.
 We show that given $\Mt=\tuple{V,\cdot_\Mt,D}$ such that $V\setminus D$ has atmost $n$ elements and
   $\vdash_{\pll} {\subseteq} \vdash_\Mt$ then 
$\Gamma_n\der_\Mt \Delta_n$.
Given $\Mt$-valuation $v$ either if $v(p_i)\in D$ 
for some  $0\leq i < n$ or $v(p_i)=v(p_j)\in V\setminus D$ for some $i\neq j$.
If the second case holds then since $p\pl p\der_{\pll} p$ we must have that $v(p_i\pl p_j)\notin V$.
Hence, if $v(\Gamma_n)\subseteq D$ then $v(\Delta_n)\cap D\neq \emptyset$.
As 
 $\{p_i\pl p_j: 0\leq i \neq j\leq n\}\not\vdash_\Mt p_i$ for $0\leq i \leq n$ we conclude that $\Mt$ is not saturated.
%
\end{proof}

It is easy to see that given two Boolean Nmatrices over disjoint signatures 
their strict product corresponds to joining the operations in a single Nmatrix.
If instead we consider the two $\Sigma_{\pll}$-Nmatrices $\BMt^\mathsf{in}_{\pll}=\tuple{\mathbf{2},\cdot_\mathsf{in},\{1\}}$ 
and $\BMt^\mathsf{el}_{\pll}=\tuple{\mathbf{2},\cdot_\mathsf{el},\{1\}}$ where
 \begin{center}
  \begin{tabular}{c | c c c }
$\pl_\mathsf{in}$ & $0$ & $1$  \\
\hline
$0$&  $ 0,1 $ & $ 0,1 $ \\
$1$ &$0,1$ & $1$ 
\end{tabular}\qquad 
  \begin{tabular}{c | c c c }
$\pl_\mathsf{el}$ & $0$ & $1$  \\
\hline
$0$&  $ 0 $ & $ 0,1 $ \\
$1$ &$0,1$ & $0,1$ 
\end{tabular}
\end{center}
and ignore the restriction on the definition of strict product above demanding that the signatures are disjoint, we obtain that  
$\BMt_{\pll}$ is isomorphic to $\BMt^\mathsf{in}_{\pll}\star \BMt^\mathsf{el}_{\pll}$ where $(0,0)$ and $(1,1)$ are renamed $0$ and $1$, respectively.
If we apply the axiomatization strategy introduced in the last subsection we obtain that 
$\der_{\BMt^\mathsf{in}_{\pll}}$ is axiomatized by the single rule $r^\scc_{\pll}$ and  $\der_{\BMt^\mathsf{in}_{\pll}}$ is axiomatized by the single rule $r^\mc_{\pll}$.
Hence, platypus' logic is the multiple-conclusion fibring of the two logics axiomatized by each of the 
rules that axiomatize $\der_\pl$, and
$\BMt_{\pll}$ is the strict product of the semantics characterizing each of these rules.
In general the strict product of two Nmatrices may end up with some entries with zero values, corresponding to entries where each of the Nmatrices output incompatible values.
There is a generalization of Nmatrices called PNmatrices where partiality is also allowed. 
We are working on a 
general description of fibred semantics 
that takes PNmatrices as semantical units.
%
%
%
%
%
%
%
%
%
%
%
%
%
%
%

\section{Complexity}\label{sec:comp}

In this section we show how moving to the richer multiple-conclusion setting may be essencial to capture an otherwise hidden behaviour of a certain connective.
Can the problem of deciding the single-conclusion logic and some of its multiple-conclusion companion differ in complexity?
It would not be hard to cook an artificial example where this is the case, however the logic of $\pl$ readily does the job.

\begin{theorem}\em \label{complexity}
 Deciding $\vdash_{\pll}$ is in $ {\bf P}$ 
and deciding $\der_{\pll}$ is {\bf coNP}-complete.
\end{theorem}
\begin{proof}
  The first part follows easily from Proposition~\ref{prop:scpl}: 
  $\nf(\Gamma)$ can be built from $\Gamma$ in polynomial time 
  and its size 
  at most as big as $\Gamma$.
  Deciding $\varphi\in L_{\pll}(\nf(\Gamma))$ can be done in polynomial time in the sum of the sizes of $\nf(\Gamma)$ and $\varphi$.
%

For the second part we will give a polynomial reduction one of the standard problems known to be {\bf NP}-complete to the problem of deciding $\not\der_{\pll}$.
 Given an instance of $3$-sat, $\mathcal{I}=\{x^i_1\ou x^i_2\ou x^i_3:1\leq i\leq k\}$ where $x^i_j$ are literals, that is $x^i_j=p^i_j$ or $x^i_j=\neg p^i_j$, 
 let:
\begin{align*}
 \Gamma_\mathcal{I}&=\{q_{x^i_1}\pl (q_{x^i_2}\pl q_{x^i_3}):1\leq i\leq k\}\\
 \Gamma_\mathsf{neg}&=\{q_{p}\pl q_{\neg p}:p\in \var(\mathcal{I})\}\\
 \Delta_\mathsf{neg}&=\{q_{\neg p}\pl q_{p}:p\in \var(\mathcal{I})\}
\end{align*}
 
 The set $\Gamma_\mathcal{I}\cup \Gamma_\mathsf{neg}\cup \Delta_\mathsf{neg}$
 can be produced in polynomial time from $\mathcal{I}$ and it is only polynomially larger than it.
Given a valuation $v$ over $\BMt$ such that $v(\Gamma_\mathcal{I})=\{1\}$ means that for each $i$ at least one of $q_{x^i_1}, q_{x^i_2},q_{x^i_3}$ must have value $1$.
 Futhermore, if $v(\Gamma_\mathsf{neg})=\{1\}$ and $v(\Delta_\mathsf{neg})=\{0\}$ means for every variable $p$ appearing in $\mathcal{I}$ either $q_{\neg p}$ or  $q_{p}$
  has value $1$ but not both.
Hence, $\mathcal{I}$ is satisfiable if and only if $\Gamma_\mathcal{I},\Gamma_\mathsf{neg}\not\der_{\pll}\Delta_\mathsf{neg}$.
%
 %
 %
\end{proof}

Is this possible   for some fragment of classical logic? The answer is \emph{no}.
Let  $\BMt$ be a  Boolean $\Sigma$-matrix where $\Sigma$ contain some arbitrary set of Boolean connectives.
From the complete characterization of the complexity of deciding the single-conclusion fragments of classical logic in \cite{BEYERSDORFF20091071}  we know
that $\vdash_\BMt$ is not {\bf coNP}-complete if and only if the connectives in $\Sigma$ are expressible using 
the constant functions $0$ and $1$, and only one of the following three binary connectives $\{\land,\lor,\oplus\}$. Deciding $\vdash_\BMt$ for all these non {\bf coNP}-complete cases is in {\bf P}.
Since $\vdash_\BMt$ is a fragment of $\der_\BMt$ and deciding the latter is always in {\bf coNP}, we have  
that  if  $\vdash_\BMt$ is  {\bf coNP}-complete then $\der_\BMt$ is also.
The complexity of $\der_{01\land}=\der_{\vdash_{01\land}}$ is in {\bf P} hence if  the connectives in $\Sigma$ are expressible using $0$, $1$ and $\land$ then $\der_{\BMt}$ is also 
in {\bf P}. 
%
Furthermore, we have that
$\Gamma\der_\BMt \{\delta_1,\ldots,\delta_k\}$ is equivalent to $\Gamma\vdash_\BMt \delta_1\lor \ldots \lor \delta_k$ wherever $\lor$ is expressible by the connectives in $\Sigma$,
and to $\Gamma, \neg\delta_2,\ldots,\neg\delta_k\vdash_\BMt \delta_1$ wherever $\neg$ is expressible by the connectives in $\Sigma$.
Easily we have that $\neg x=p\oplus 1$. 
Hence, 
in all
 remaining cases deciding $\der_\BMt$ is polynomially reducible to either $\vdash_{01\lor}$ or $\vdash_{01\oplus}$ and are therefore also in {\bf P}.

 In \cite{BEYERSDORFF20091071} some (sub-polynomial) nuances between deciding the single-premise fragments of $\vdash_\BMt$ are also studied. 
Most likely, such variations are also present in the divide $\vdash_\BMt$ and $\der_\BMt$ but we leave such analysis for another occasion.

\section{Conclusion} 
\label{sec:final}

We have shown how incorporating non-determinism in semantics and moving towards a symmetrical view on logic, by considering multiple-conclusion consequence relations,
significantly widens the range of available tools. 
In particular, it allows for a compositional, bottom up, approach to the analysis of logics.
Despite the recent advances there is still a long way to go to reach the depth and sophistication of what is known in the particular field of modal logics.
We look to it as source of case studies and inspiration on where to go. 
%
We intend to improve on \cite{wollic} and
provide a completely general and modular semantics for combined logics by taking as semantical units PNmatrices, covering, of course, what is known for fusion of modal logics.
 We expect to take profit from such knowledge, and develop techniques as in \cite{CaMaAXS},
  to provide insights on the semantics of strengthening of logics with sets of axioms depending on their 
  syntactical structure, akin to what is known regarding Sahlqvist formulas.

 The use of
 the extra expressivity of (P)Nmatrix semantics to capture relevant behaviours in many practical engineering or scientific
 contexts is still in its infancy.
 Nmatrices have been used already to give effective semantics to a big range of important non-classical logics 
 \cite{Baaz2013,av2,coniglioswap,ISMVL}.
 These structures can for example be used to deal with paraconsistent behaviour \cite{taming}, to model how a processor deals with information from multiple-sources \cite{Avron}, or for reasoning about computation errors \cite{DBLP:journals/sLogica/AvronK09}. The fact that $\pl$ can be seen as a defective $\land$ or $\lor$ might indicate that (P)Nmatrices can
 be of value when reasoning about unreliable logic circuits \cite{ASJRCSPM14,ASJRCSPM17}.
 Also, recently a natural interpretation of quantum states as valuations over Nmatrices was introduced \cite{quantum}.

 
  There is still another front opened by considering multi-functions. 
  Acknowledging the importance of the study of function clones for logic, suggests that the development of the theory of multi-function clones might be 
  an interesting path to follow.
The fact that there are various possible notions of composition, relating to the differences between tree-automata, term-automata and dag-automata 
\cite{tata2007,Siva-Dran-Rusi05b}, 
complicates the question.
However the connections with such well established and studied structures looks promising.

We finish with a table summarizing some of the results obtained in this paper, illustrating the differences between the logic of $\pl$ and 
the classical fragments with $\land$ and $\lor$. 
We can see that certain properties of $\land$ and $\lor$ are preserved (being inclusion logics for instance).
However, others are not. Finite-valuedness in terms of matrices and finite axiomatizability using single-conclusion rules are lost. In terms of complexity, deciding the Tarskian logic defined by $\pl$ is in the same complexity of each $\land$ and $\lor$ separately. However deciding its Scottian logic meets the maximum possible complexity of deciding a logic given by a finite Nmatrix, coinciding with the complexity of deciding any fragment of classical logic expressing both $\land$ and $\lor$.
Lastly, it seems interesting to explore the relation of $\pl$, and the process that gave rise to it,  with the notion of meet-combination of logics 
 introduced by 
 Sernadas et al. in \cite{10.1093/logcom/exr035}.
\begin{center}
 \begin{tabular}{c | c | 
  c | c | c | c }
 {\small Logic} &  {\small Fin. Ax.} 
&  {\small Fin. matrix} & 
  {\small Complexity} & {\small Left-inclusion} &  {\small Right-inclusion} 
\\
\hline
$\vdash_\land$& Y & Y & {\bf  P} & N& Y\\
$\vdash_\lor$& Y & Y & {\bf  P} & Y& N\\
$\vdash_{\pll}$& N & N & {\bf  P} & Y& Y\\
$\vdash_{\land\lor}$& Y & Y & {\bf  coNP} & N& N\\
$\der_\land$& Y & Y & {\bf  P} & N& Y\\
$\der_\lor$& Y & Y & {\bf  P} & N& Y\\
$\der_{\pll}$& Y & N & {\bf  coNP} & Y& Y\\
$\der_{\land\lor}$& Y & Y & {\bf  coNP} & N& N\\
\hline
\end{tabular}
\end{center}

\bibliographystyle{plain}  
\bibliography{biblio.bib}

{ \color{white} 
 { \section*{\tiny Hidden track}}
 ~\\[-1.3cm]
\begin{center}
 {\tiny
 {\bf Ornithorhynchus anatinus - platypus}\smallskip

Sleeping geology\\
On the isolated shore\\
For millions of years

Experimental continent\\
On purpose or accident?\\

Mysterious evolving\\
Problem solving\\

A vaudeville? A nation including one superior creation\\
A vertebra? Inverted...quite unheard of...\\

Orphan in a family\\
And a sole survivor\\
He's a living fossil

Reptillian? Mammalian\\
He's a bird-beaked, beaver-butt Australian\\

Amphibious? Paradox wearing plaid socks\\
Furry beetle? A bugbear, and a palezoologist's nightmare\\
Symmetrical physique of disbelief\\

The platypus has the brain of a dolphin\\
and can be seen driving a forklift in his habitat of kelp\\
He is the larva of the flatworm\\
and has the ability to regenerate after injury

No relation to the flounder.

Someone shipped him to the blokes\\
Who said he was a hoax\\
So they cut him to pieces, wrote a thesis\\

A cranium of deceit, he's prone to lie and cheat;\\
It's no wonder -- a blunder from down under

Duckbill, watermole, duckmole!

Barnacle

\hspace{4cm} {\bf Mr. Bungle}}
\end{center}
}

\end{document}